\documentclass[11pt, oneside]{amsart}

\usepackage{amsmath,amssymb,verbatim,amsthm,bbm, mathrsfs,amscd,pb-diagram}
\usepackage{aliascnt} 
\usepackage{appendix}
\DeclareMathAlphabet{\mathpzc}{OT1}{pzc}{m}{it}
\usepackage[all]{xy}
\usepackage{multirow,longtable}
\usepackage{diagbox}
\usepackage{makecell}
\usepackage{float} 
\usepackage{hyperref}
\usepackage{cleveref}
\usepackage{marginnote}

\usepackage{framed}

\usepackage{csquotes}

\usepackage{tikz}

\usepackage{multicol}

\usepackage{lscape} 

\usepackage{framed}

\usepackage{dynkin-diagrams}

\usepackage{cite}

\usepackage[disable]{todonotes}

\newcounter{todocounter}
\newcommand{\todonum}[1]{\stepcounter{todocounter}\todo{\thetodocounter: #1}}
\makeatletter
\providecommand\@dotsep{5}
\renewcommand{\listoftodos}[1][\@todonotes@todolistname]{%
	\@starttoc{tdo}{#1}}
\makeatother

\usepackage{float}
\restylefloat{table}

\crefname{table}{table}{tables}
\crefname{listing}{Program-code}{Program-codes}  
\Crefname{listing}{Program-code}{Program-codes}
\crefname{subsection}{subsection}{subsections}

\theoremstyle{plain}

\newtheorem{Thm}{Theorem}[section]
\newtheorem{Cor}[Thm]{Corollary}
\newtheorem{Prop}[Thm]{Proposition}
\newtheorem{Lem}[Thm]{Lemma}

\theoremstyle{definition}
\newtheorem{Remark}[Thm]{Remark}

\numberwithin{equation}{section}
\newcommand{\Card}[1]{\left\vert #1\right\vert} 
\newcommand{\coset}[1]{\left[ #1 \right]}  
\newcommand{\gen}[1]{\left\langle #1 \right\rangle}  
\newcommand{\FNorm}[1]{\left\vert #1 \right\vert} 

\newcommand{\Ind}{\operatorname{Ind}}

\newcommand{\Hom}{\operatorname{Hom}}

\newcommand{\PGL}{\operatorname{PGL}}

\newcommand{\SL}{\operatorname{SL}}
\newcommand{\GL}{\operatorname{GL}}

\newcommand{\C}{\mathbb C}

\newcommand{\Q}{\mathbb{Q}}

\newcommand{\Z}{\mathbb{Z}}
\newcommand{\R}{\mathbb{R}}
\newcommand{\N}{\mathbb{N}}

\newcommand{\bk}[1]{\left(#1\right)} 
\newcommand{\bm}{\begin{multline*}}
\newcommand{\tu}{\end  {multline*}}

\DeclareMathOperator{\Id}{\mathbf{1}} 
\newcommand{\Gm}{\mathbb{G}_m} 

\renewcommand{\check}[1]{#1 ^{\vee}} 
\DeclareMathOperator{\Real}{\mathfrak{Re}} 
\DeclareMathOperator{\Imaginary}{\mathfrak{Im}} 
\newcommand{\piece}[1]{\left\{\begin{matrix} #1 \end{matrix}\right.} 
\newcommand{\set}[1]{\left\{ #1 \right\}} 
\newcommand{\mvert}{\mathrel{}\middle\vert\mathrel{}} 

\newcommand{\suml}{\sum\limits}

\newcommand{\rmod}{/}
\newcommand{\lmod}{\backslash}
\newcommand{\Stab}{\operatorname{Stab}}

\newcommand{\esixchar}[6]{\renewcommand*{\arraystretch}{1} \begin{pmatrix}&& #2 && \\ #1 & #3 & #4 & #5 & #6 \end{pmatrix} }

\newcommand{\esixcharclass}[6]{\renewcommand*{\arraystretch}{1} \begin{bmatrix}&& #2 && \\ #1 & #3 & #4 & #5 & #6 \end{bmatrix} }

\newcommand{\bfM}{\mathbf{M}}

\newcommand{\bfT}{\mathbf{T}}

\newcommand{\bfX}{\mathbf{X}}

\newcommand{\sepline}{\mbox{}\linebreak\noindent\rule{\textwidth}{2pt}\linebreak}

\makeatletter
\newcommand*{\rom}[1]{\expandafter\@slowromancap\romannumeral #1@}
\makeatother

\renewcommand*{\arraystretch}{1.5}

\makeatletter
\def\imod#1{\allowbreak\mkern10mu({\operator@font mod}\,\,#1)}
\makeatother

\makeatletter
\renewcommand\section{\@startsection{section}{1}{\z@}%
	{-3.5ex \@plus -1ex \@minus-.2ex}%
	{2.3ex \@plus.2ex}%
	{\center\normalfont\large\bfseries}}
\makeatother

\makeatletter
\renewcommand\subsection{\@startsection{subsection}{2}{\z@}%
	{-3.5ex \@plus -1ex \@minus-.2ex}%
	{2.3ex \@plus.2ex}%
	{\normalfont\large\bfseries}}
\makeatother

\makeatletter
\renewcommand\subsubsection{\@startsection{subsubsection}{3}{\z@}%
	{-3.5ex \@plus -1ex \@minus-.2ex}%
	{2.3ex \@plus.2ex}%
	{\normalfont\large\bfseries}}
\makeatother

\makeatletter
\newtheorem*{rep@theorem}{\rep@title} \newcommand{\newreptheorem}[2]{%
	\newenvironment{rep#1}[1]{%
		\def\rep@title{\bf #2 \ref{##1} }%
		\begin{rep@theorem} }%
		{\end{rep@theorem} } }
\makeatother
\newreptheorem{theorem}{Theorem}
\newreptheorem{lemma}{Lemma}

\protected\def\ignorethis#1\endignorethis{}
\let\endignorethis\relax

\usepackage{colortbl}
\usepackage{zref-savepos}
\usepackage{pict2e}

\newcounter{NoTableEntry}
\renewcommand*{\theNoTableEntry}{NTE-\the\value{NoTableEntry}}

\makeatletter
\newcommand*{\notableentry}{%
	\kern-\tabcolsep
	\stepcounter{NoTableEntry}%
	\vadjust pre{\zsavepos{\theNoTableEntry t}}
	\vadjust{\zsavepos{\theNoTableEntry b}}
	\zsavepos{\theNoTableEntry l}
	\raisebox{%
		\dimexpr\zposy{\theNoTableEntry b}sp
		-\zposy{\theNoTableEntry l}sp\relax
	}[0pt][0pt]{%
		\setlength{\unitlength}{1pt}%
		\edef\w{%
			\strip@pt\dimexpr\zposx{\theNoTableEntry r}sp%
			-\zposx{\theNoTableEntry l}sp\relax
		}%
		\edef\h{%
			\strip@pt\dimexpr\zposy{\theNoTableEntry t}sp%
			-\zposy{\theNoTableEntry b}sp\relax
		}%
		\ifdim\w pt=0pt 
		\else
		\begin{picture}(0,0)%
		\edef\x{%
			\noexpand\put(0,0){\noexpand\line(\w,\h){\w}}%
			\noexpand\put(0,\h){\noexpand\line(\w,-\h){\w}}%
		}\x
		\end{picture}%
		\fi
	}%
	\hspace{0pt plus 1filll}%
	\zsavepos{\theNoTableEntry r}
	\kern-\tabcolsep
}

\makeatletter
\providecommand*{\cupdot}{%
	\mathbin{%
		\mathpalette\@cupdot{}%
	}%
}
\newcommand*{\@cupdot}[2]{%
	\ooalign{%
		$\m@th#1\cup$\cr
		\sbox0{$#1\cup$}%
		\dimen@=\ht0 %
		\sbox0{$\m@th#1\cdot$}%
		\advance\dimen@ by -\ht0 %
		\dimen@=.5\dimen@
		\hidewidth\raise\dimen@\box0\hidewidth
	}%
}

\providecommand*{\bigcupdot}{%
	\mathop{%
		\vphantom{\bigcup}%
		\mathpalette\@bigcupdot{}%
	}%
}
\newcommand*{\@bigcupdot}[2]{%
	\ooalign{%
		$\m@th#1\bigcup$\cr
		\sbox0{$#1\bigcup$}%
		\dimen@=\ht0 %
		\advance\dimen@ by -\dp0 %
		\sbox0{\scalebox{2}{$\m@th#1\cdot$}}%
		\advance\dimen@ by -\ht0 %
		\dimen@=.5\dimen@
		\hidewidth\raise\dimen@\box0\hidewidth
	}%
}
\makeatother

\allowdisplaybreaks

\title[Degenerate Principal Series of $E_6$]{The Structure of Degenerate Principal Series Representations of Exceptional Groups of Type $E_6$ over $p$-adic Fields} 
\author[Hezi Halawi and Avner Segal]{Hezi Halawi${^{1}}$ and Avner Segal${^{2,3}}$}
\address{${^1}$ School of Mathematics, Ben Gurion University of the Negev, POB 653, Be'er Sheva 84105, Israel}
\address{${^2}$ Department of Mathematics, University of British Columbia, Vancouver, BC, V6T 1Z2, Canada}
\address{${^3}$ Department of Mathematics, Bar Ilan University, Ramat Gan 5290002, Israel}
\email{halawi@math.bgu.ac.il, avners@math.ubc.ca, segalav4@biu.ac.il}

\numberwithin{equation}{section}

\subjclass[2010]{22E50, 20G41, 20G05}

\begin{document}

\maketitle

\begin{abstract}
	In this paper, we study the reducibility of degenerate principal series of the simple, simply-connected exceptional group of type $E_6$.
	Furthermore, we calculate the maximal semi-simple subrepresentation and quotient of these representations.
\end{abstract}

\tableofcontents

\section{Introduction}

\todonum{Grammar and spell check. Acknowledgments.}

One of the main problems in the representation theory of $p$-adic groups is the question of reducibility and structure of parabolic induction.
More precisely, let $G$ be a $p$-adic group and let $P$ be a parabolic subgroup of $G$.
Let $M$ denote the Levi subgroup of $P$ and let $\sigma$ be a smooth irreducible representation of $M$.
One can ask:
\begin{itemize}
	\item Is the parabolic induction $\pi=\Ind_P^G\sigma$ reducible?
	\item If the answer to the previous question is positive, what is Jordan-H\"older series of $\pi$?
\end{itemize}


In this paper we completely answer the first question for degenerate principal series of the exceptional group of type $E_6$.
We further calculate the maximal semi-simple subrepresentation and quotient of such representations which partially answers the second question.

More precisely, let $G$ be a simple, split, simply-connected $p$-adic group of type $E_6$ and let $P$ be a maximal parabolic subgroup of $G$.
For a one-dimensional representation $\Omega$ of $P$, we consider the normalized parabolic induction $\pi=\Ind_P^G\Omega$.
We determine for which $\Omega$, $\pi$ is reducible and its maximal semi-simple subrepresentation and maximal semi-simple quotient.
Our main result, \Cref{Thm:Deg_PS_of_E6}, is summarized by the following corollary.

\begin{Cor}
	With the exception of one case, for any maximal Levi subgroup $M$ of $G$ (of type $E_6$) and any one-dimensional representation $\Omega$ of $M$, if the degenerate principal series representation $\pi=\Ind_P^G\Omega$ is reducible, then $\pi$ admits a unique irreducible subrepresentation and a unique irreducible quotient.

	In the one exception, up to contragredience, $\Omega=\bk{\chi\circ\omega_4}\FNorm{\omega_4}^{1/2}$, where $\omega_4$ is the $4^{th}$ fundamental weight and $\chi$ is a cubic character.
	In this case, $\pi$ admits a maximal semi-simple quotient of length $3$ and a unique irreducible subrepresentation.
\end{Cor}


Similar studies were performed for both classical and smaller exceptional groups.
For example, see:
\begin{itemize}
	\item \cite{MR584084,MR0425030} for general linear groups. (It should be noted that the scope of these works goes beyond degenerate principal series)
	\item \cite{MR1134591} for symplectic groups.
	\item \cite{MR2017065,MR1346929} for orthogonal groups.
	\item \cite{MR1480543} for type $G_2$.
	\item \cite{MR2778237} for type $F_4$.
\end{itemize}

The reason that such a study was not preformed for groups of type $E_n$ before, is that Weyl groups of these types are extremely big and have complicated structure.
For that reason, in \Cref{Sec_algorithm} we describe an algorithm, to calculate the reducibility of degenerate principal series.
This algorithm was implemented by us using \emph{Sagemath} (\cite{sagemath}).

The study of degenerate principle series for groups of type $E_7$ and $E_8$, using the algorithm described in this paper, is a work in progress.

This paper is organized as follows:
\begin{itemize}
	\item \Cref{Sec_Preliminaries} introduces the notations used in this paper.
	
	\item \Cref{Sec_algorithm} introduces the algorithm used by us to calculate the reducibility of $\pi=\Ind_P^G\Omega$ and the maximal semi-simple subrepresentation and quotient of $\pi$.
	
	\item \Cref{Sec_degenerate_principal_series_E6} begins with a short introduction of the group of $G$ and contains our main result, \Cref{Thm:Deg_PS_of_E6}, the reducibility of degenerate principal series of $E_6$, their subrepresentations and their quotients.
	
	\item \Cref{App:Branching_Rules} contains a few technical results which are useful in the implementation of the algorithm in \Cref{Sec_algorithm}.
	
	\item \Cref{App:Example_of_Irr_Cir} contains an example of the application of the irreducibility test described in \Cref{Sec_algorithm}.
\end{itemize}

\todonum{Maybe add something about the residual spectrum to the introduction?}

%

{\bf Acknowledgments.} 
Both authors wish to thank Nadya Gurevich for many helpful discussion during the course of this work.

This research was supported by the Israel Science Foundation, grant number
421/17 (Segal).

\section{Preliminaries and Notations}
\label{Sec_Preliminaries}
Let $k$ be a locally-compact non-Archimedean field and $G$ the group of $k$-rational points of a split simply-connected reductive algebraic group over $k$.
Fix a Borel subgroup $B$ of $G$ with torus $T$ and assume that $n=rank\bk{G}$.
Let $\Phi$ denote the set of roots of $G$ with respect to $T$ and let $\Delta$ denote the set of simple roots with respect to $B$.
Also, let $\Phi^{+}$ denote the set of positive roots of $G$ with respect to $B$.
For $\alpha\in\Delta$, let $\check{\alpha}$ denote the associated co-character and let $\omega_\alpha$ denote the associated fundamental weight.
Namely,
\[
\gen{\omega_\alpha,\check{\beta}} = \delta_{\alpha,\beta} \quad \forall \alpha,\beta\in\Delta ,
\]
where $\delta_{\alpha,\beta}$ denotes the usual Kronecker delta function on $\Delta$.

Let $W=W_G$ denote the Weyl group of $G$ with respect to $T$, it is generated by simple reflections $s_\alpha$ with respect to the simple roots $\alpha\in\Delta$

For $\Theta\subset\Delta$, let $P_\Theta$ denote the associated standard parabolic subgroup of $G$ and denote its Levi subgroup by $M_\Theta$.
In this case, let $\Delta_M=\Theta$.
Let $\Phi_M = \Phi\cap Span_\Z\bk{\Delta_M}$ be the set of roots of $M$ with respect to $T$, $\Phi_M^{+}=\Phi^{+}\cap\Phi_M$ be the set of positive roots of $M$ and let
\[
\rho_M=\rho_M^G=\frac{1}{2} \suml_{\alpha\in\Phi_M^{+}} \alpha
\]
be the half-sum of positive roots in $M$.
Also, let $W_M$ denote the Weyl group of $M$ relative to $T$, it is the subgroup of $W_G$ generated by $\set{s_\alpha\mvert \alpha\in\Delta_M}$.

For any $\alpha\in\Delta$ we denote $L_\alpha=M_{\set{\alpha}}$ and $M_\alpha=M_{\Delta\setminus\set{\alpha}}$.


For a Levi subgroup $M$ of $G$, we denote its complex manifold of characters by $\bfX\bk{M}$.
As $\bfX\bk{M}$ is a commutative group, it is convenient to adopt additive notations for its elements.
In particular, for $\lambda\in\bfX\bk{M}$, $-\lambda$ denotes its inverse.
The trivial element in $\bfX\bk{M}$ will be denoted by either $\Id$ or $0$.

\todonum{Think if there is a better way of presenting the different spaces of characters. In particular, think about the equation $\bfX\bk{M} = \mathfrak{a}^\ast_{M,\C} \oplus \bfX_{M,0}$.}

Within $\bfX\bk{M}$, we distinct the \emph{real and imaginary characters} on $M$,
\[
\Real\bk{\bfX\bk{M}} = \set{\lambda \in \bfX\bk{M} \mvert \gen{\lambda,\check{\alpha}}\in\R \ \forall \alpha\in\Delta},\quad \Imaginary\bk{\bfX\bk{M}} = i \Real\bk{\bfX\bk{M}} .
\]
It holds that $\bfX\bk{M} = \Real\bk{\bfX\bk{M}} \oplus \Imaginary\bk{\bfX\bk{M}}$.
For $\mu\in \bfX\bk{M}$, we write $\Real\bk{\mu}$ for the projection of $\mu$ to $\Real\bk{\bfX\bk{M}}$ and $\Imaginary\bk{\mu}$ for the projection of $\mu$ to $\Imaginary\bk{\bfX\bk{M}}$.

Let $\mu\in \bfX\bk{M}$, and assume that there exists $k\in\N$ such that $\mu^k=\Id$.
We say that $\mu$ is of \emph{finite order} and denote the minimal such $k$ by $ord\bk{\mu}$, the \emph{order of $\mu$}.
In particular, $ord\bk{\Id}=1$ and, for a quadratic character $\mu$, $ord\bk{\mu}=2$.

We further denote by $X^\ast\bk{M}$ the ring of rational characters of $M$ and let $\mathfrak{a}_{M,\C}^\ast=X^\ast\bk{M}\otimes_\Z\C$.
One can choose a direct sum complement $\bfX\bk{M} = \mathfrak{a}^\ast_{M,\C} \oplus \bfX_{M,0}$, where the characters in $\bfX_{M,0}$ are of finite order.

We note that restriction to $T$ gives rise to the following inclusions
\begin{equation}
\label{Eq:Inculsions_of_character_groups}
\iota_M:\bfX\bk{M} \hookrightarrow \bfX\bk{T}, \quad
X^\ast\bk{M}\hookrightarrow X^\ast\bk{T}, \quad
\mathfrak{a}^\ast_{M,\C}\hookrightarrow \mathfrak{a}^\ast_{T,\C}, \quad
\bfX_{M,0} \hookrightarrow \bfX_{T,0}.
\end{equation}

The image of these embeddings can be identified by restriction to the derived group $M^{der}$.
Namely, for $\chi\in \bfX\bk{T}$ it holds that $\chi\in \bfX\bk{M}$ if and only if $\gen{\chi,\check{\alpha}}=0$ for all $\alpha\in\Delta_M$.
Similarly for $X^\ast\bk{M}$, $\mathfrak{a}^\ast_{M,\C}$ and $\bfX_{M,0}$.

Recall that $\mathfrak{a}^\ast_{M,\C}\cong \C^{n-\Card{\Delta_M}}$.
We use this identification and write elements of $\mathfrak{a}^\ast_{M,\C}$ as vectors in $\C^{n-\Card{\Delta_M}}$ whose basis elements are the fundamental weights $\omega_\alpha$ for $\alpha\in\Delta\setminus\Delta_M$.

In fact, by identify $\bfX\bk{\Gm\bk{k}}$ with $\C$ we may write any element $\Omega\in\bfX\bk{M}$ in the form
\[
\Omega = \suml_{\alpha\in\Delta\setminus\Delta_\bfM} \Omega_\alpha\circ \omega_\alpha,
\]
where $\Omega_\alpha\in\bfX\bk{\Gm\bk{k}}$.

Let $\lambda\in \bfX\bk{T}$.
We say that $\lambda$ is \emph{dominant} if
\[
\Real\bk{\gen{\lambda,\check{\alpha}}} \geq 0 \quad \forall \alpha\in\Delta .
\]
In other words, $\Real\bk{\lambda}$ is in the (closed) positive Weyl chamber.
We say that $\lambda$ is \emph{anti-dominant} if $-\lambda$ is dominant.

Let $M$ be the Levi subgroup of a parabolic subgroup $P$ of $G$.
We recall two functors associating representations of $M$ and $G$.
\begin{itemize}
	\item For a representation $\Omega$ of $M$, we denote the normalized parabolic induction of $\Omega$ to $G$ by $i_M^G\Omega$.
	In particular, $i_M^G:Rep\bk{M} \to Rep\bk{G}$.
	
	\item For a representation $\pi$ of $G$, we denote the normalized Jacquet functor of $\pi$ along $P$ by $r_M^G\pi$.
	In particular, $r_M^G:Rep\bk{G} \to Rep\bk{M}$.
\end{itemize}
These two functors are adjunct to each other, namely they satisfy the Frobenius reciprocity
\begin{equation}
\Hom_G\bk{\pi,i_M^G\sigma} \cong \Hom_M\bk{r_M^G\pi,\sigma} .
\end{equation}

Note that for a 1-dimensional representation $\Omega$ of $M$, it holds that
\[
r_T^M\Omega = \iota_M\bk{\Omega} - \rho_{B\cap M}^M .
\]

In most of this work, we consider representations of a group $H$ as elements in the Grothendieck ring $\mathfrak{R}\bk{H}$ of $H$.
In particular, we write $\pi=\pi_1+\pi_2$ if, for any irreducible representation $\rho$ of $H$, it holds that
\[
mult\bk{\rho,\pi} = mult\bk{\rho,\pi_1} + mult\bk{\rho,\pi_2}.
\]
Here, $mult\bk{\rho,\pi}$ denotes the multiplicity of $\rho$ in the Jordan-H\"older series of $\pi$.
Furthermore, we write $\pi\leq \pi'$ if, for any irreducible representation $\rho$ of $H$, it holds that
\[
mult\bk{\rho,\pi} \leq mult\bk{\rho,\pi'}.
\]

\begin{Remark}
	Both $\bfX\bk{M}\subset Rep\bk{M}$ and $\mathfrak{R}\bk{M}$ are commutative groups of interest here.
	Furthermore, there is a natural map
	\[
	Rep\bk{M} \longrightarrow \mathfrak{R}\bk{M}.
	\]
	And so, elements of $\bfX\bk{M}$ can be identified with elements in $\mathfrak{R}\bk{M}$.
	However, this is not a group homomorphism.
	Namely, the sum of two characters in $\bfX\bk{M}$ is a character in $\bfX\bk{M}$, while a sum of two characters in $\mathfrak{R}\bk{M}$ is a two-dimensional representation of $M$.
	
	Hence, in order to avoid confusion, given $\Omega\in Rep\bk{M}$ we will write $\coset{\Omega}$ for the image of  in $\mathfrak{R}\bk{M}$.
	In that spirit, we have
	\begin{itemize}
		\item $\Omega_1+\Omega_2\in\bfX\bk{M}$ is the product of $\Omega_1$ and $\Omega_2$ in $\bfX\bk{M}$, i.e. $\bk{\Omega_1+\Omega_2}\bk{m} = \Omega_1\bk{m}\Omega_2\bk{m}$.
		\item $n\cdot\Omega \in \bfX\bk{M}$ is the character given by $\bk{n\cdot \Omega}\bk{m} = \Omega\bk{m^n}$.
		\item $-\Omega$ is the inverse of $\Omega \in \bfX\bk{M}$.
	\end{itemize}
	While, on the other hand,
	\begin{itemize}
		\item $\coset{\Omega_1}+\coset{\Omega_2}$ is the class, in $\mathfrak{R}\bk{M}$, of the direct sum $\Omega_1\oplus\Omega_2$.
		\item $n\times\coset{\Omega} \in \mathfrak{R}\bk{M}$ is the class in $\mathfrak{R}\bk{M}$ of $\oplus_{i=1}^n \Omega$.
		\item $-\coset{\Omega}$, with $\Omega \in Rep\bk{M}$, is a virtual representation.
	\end{itemize}
\end{Remark}

We recall the following fundamental result (see \cite[lem. 2.12]{MR0579172} \cite[Theorem 6.3.6]{Casselman}).
\begin{Lem}[Geometric Lemma]
	For Levi subgroups $L$ and $M$ of $G$, let
	\[
	W^{M,L} = \set{w\in W\mvert w\bk{\Phi^+_M}\subset\Phi^+,\ w^{-1}\bk{\Phi^+_L}\subset\Phi^{+}}
	\]
	be the set of shortest representatives in $W$ of $W_L\lmod W\rmod W_M$.
	For a smooth representation $\Omega$ of $M$, the Jacquet functor $r_L^G i_M^G \Omega$, as an element of $\mathfrak{R}\bk{L}$, is given by:
	\begin{equation}
	\coset{r_L^G i_M^G \Omega} = \sum_{w\in W^{M,L}} \coset{i_{L'}^L \circ w \circ r_{M'}^M \Omega} ,
	\end{equation}
	where, for $w\in W^{M,L}$, 
	\begin{align*}
	& M'=M\cap w^{-1}\bk{L} \\
	& L'=w\bk{M}\cap L .
	\end{align*}
\end{Lem}

In particular, for an admissible representation $\Omega$ of $M$, the Jacquet functor $r_T^M\Omega$, as an element of $\mathfrak{R}\bk{T}$, is a finite sum of one-dimensional representations of $T$.
Each such representation is called an \emph{exponent} of $\Omega$.
If $\lambda=r_T^M\Omega$ is one-dimensional, then $\lambda$ is said to be the \emph{leading exponent} of $i_M^G\Omega$.

Let $M=M_{\Delta\setminus\set{\alpha}}$ denote the Levi subgroup of a standard maximal parabolic subgroup $P$ of $G$.
Any element in $\bfX\bk{M}$ is of the form 
\begin{equation}
\Omega_{M,\chi,s} = \bk{s+\chi}\circ\omega_\alpha,
\end{equation}
where $s\in\C$ and $\chi\in\bfX_{M,0}$.
We denote $\chi_{M,s}=r_T^G \Omega_{M,\chi,s}$ and
\[
I_P\bk{\chi,s} = i_M^G \Omega_{M,\chi,s} .
\]
Note that
\begin{equation}
\label{Eq_I_P_injects_to_I_B}
I_P\bk{\chi,s} \hookrightarrow i_T^G \chi_{M,s} .
\end{equation}


In this paper we study the degenerate principal series representations of a the simply-connected, simple, split group $G$ of type $E_6$.
Namely, we study the reducibility of $I_P\bk{\chi,s}$, when $P$ is a maximal parabolic subgroup of $G$.
Since the algorithm described in \Cref{Sec_algorithm} is applicable to other groups, we postpone the discussion on the structure of this group to \Cref{Sec_degenerate_principal_series_E6}.

\section{Method Outline}
\label{Sec_algorithm}
In this section we explain the algorithm used to compute the reducibility of degenerate principal series and the mathematical background behinds it.

Through out this section, we fix $G$.
For a maximal parabolic subgroup $P$ of $G$ with Levi subgroup $M$, $s\in\C$ and $\chi\in\bfX_{M,0}\cong\bfX_{GL_1,0}$ let $\pi=I_P\bk{\chi,s}$.
For any such triple $\bk{P,s,\chi}$ we wish to determine the reducibility of $I_P\bk{\chi,s}$ and the structure of its maximal semi-simple subrepresentation and quotient.
The algorithm we implement for doing this has four parts:
\begin{enumerate}
	\item 
	For regular $I_P\bk{\chi,s}$ determine whether $I_P\bk{\chi,s}$ is reducible or not.
	This part is completely determined by the results of \cite{Casselman,MR0579172}.
	\item Apply various reducibility criteria for non-regular $I_P\bk{\chi,s}$.
	\item Apply an irreducibility test for non-regular $I_P\bk{\chi,s}$ by applying a chain of \emph{branching rules} (which will be introduced in \Cref{Subsec:Irr_Test}).
	\item
	For a reducible representation, determine its maximal semi-simple subrepresentation and quotient.
\end{enumerate}

For the convenience of the reader, in framed boxes throughout this section, we outline the way in which we implemented the relevant calculations.

Also, we wish to point out that, as it turns out, the results of the following calculations depend on $ord\bk{\chi}$ but not the particular choice of $\chi$.


\subsection{Part \rom{1} - Regular Case}

We say that $I_P\bk{\chi,s}$ is \emph{regular} if $\Stab_W\bk{\chi_{M,s}}=\set{e}$.

\begin{Remark}
	We note the following facts regarding stabilizers of elements in $\bfX\bk{M}$:
	\begin{itemize}
		\item The stabilizer of a character $\lambda\in \mathfrak{a}_{T,\R}^\ast$ is generated by the generalized reflections $s_\beta$ for $\beta\in\Phi$ orthogonal to $\lambda$.
		
		\item If $I_P\bk{\chi,s}$ is non-regular, then the imaginary part must satisfy $\Imaginary\bk{s}\in 2\pi\cdot\Q	$.
		Namely, $\Imaginary\bk{s}$ (as an element of $\bfX\bk{GL_1}$) has finite order.
		In particular, one can assume that $s\in\R$.
		
		\item Since $\Stab_W\bk{\chi_{M,s}} = \Stab_W\bk{\Real\bk{\chi_{M,s}}} \cap \Stab_W\bk{\Imaginary\bk{\chi_{M,s}}}$, it follows that if $I_P\bk{\chi,s}$ is non-regular, then $\Stab_W\bk{\Real\bk{\chi_{M,s}}}$ is non-trivial.
		
		\item Given $M$ and $s\in\R$, $\Stab_W\bk{\chi_{M,s}}$ depends only on the order of $\chi$.
		\item Note that there are only finitely many triples $\bk{P,\chi,s}$ such that $I_P\bk{\chi,s}$ is non-regular.
	\end{itemize}
\end{Remark}


\begin{framed}
	\textbf{Calculating non-regular points:}
	
	In order to find all non-regular $I_P\bk{\chi,s}$ we proceed with the following steps:
	
	\begin{itemize}
		\item We start by making a list $X$ all values of $s$ such that $\Stab_W\bk{\Real\bk{\chi_{M,s}}}$ is non-trivial.
		We do this be listing all $s\in\R$ such that $\Real\bk{\gen{\chi_{M,s},\check{\alpha}}}=0$ for some $\alpha\in\Phi^{+}$.
		
		\item We then make a list $Y$ of all values $m\in\N$ such that $\Stab_W\bk{\Imaginary\bk{\chi_{M,s}}}$ is non-trivial.
		For any $w \in W^{M,T}$, we write $w\cdot \omega_\alpha -\omega_\alpha =\sum_{\alpha \in \Delta_{G}}{n_{\alpha}\alpha}$ and let $m_0 = gcd \bk{n_{\alpha}}_{\alpha \in \Delta_{G}}$.
		We list all the positive divisors $m$ of $m_0$.
		
		\item For all candidates $\bk{s,m}\in X\times Y$, we check if $\chi_{M,s}$ is non-regular, where $\chi$ is of order $m$.
		\begin{itemize}
			\item Let $w\in W$ such that $w\cdot\Real\bk{\chi_{M,s}}$ is anti-dominant and denote $\lambda_{a.d.}=w\cdot\chi_{M,s}$.
			\item The stabilizer of $\Real\bk{\lambda_{a.d.}}$ is generated by the simple reflections associated with $\alpha\in\Delta$ such that $\gen{\Real\bk{\lambda_{a.d.}},\check{\alpha}}=0$.
			\item The imaginary part of $\lambda_{a.d.}$ is given by $\Imaginary\bk{\lambda_{a.d.}}=\chi\circ\bk{w\cdot\omega_i}$.
			\item
			If there exits a non trivial $u \in  \Stab_W\bk{\Real\bk{\chi_{M,s}}}$ such that $\chi \circ \bk{w \cdot \omega_i} =  \chi \circ \bk{uw \cdot \omega_i}$ then $I_{P}\bk{\chi,s}$ is non-regular.
		\end{itemize}
	\end{itemize}

\end{framed}

For $\pi=I_P\bk{\chi,s}$ and a Levi subgroup $L$ of $G$ we let its \emph{Bernstein-Zelevinsky set} to be
\begin{equation}
BZ_L\bk{\pi} = \set{i_{L'}^L\circ w\circ\tau \mvert 
	\begin{matrix}
	w\in W^{M,L} \\
	M'=M\cap w^{-1}\bk{L} \\
	L'=w\bk{M}\cap L \\
	\tau\leq r_{M'}^M\bk{\Omega_{M,\chi,s}} \text{ is irreducible} \\
	\end{matrix}
}
\end{equation}


We quote a corollary of \cite[Theorem 3.1.2]{MR1134591} for this case.
\begin{Cor}
	\label{Thm:BZ-components}
	The following are equivalent:
	\begin{itemize}
		\item $\pi=I_P\bk{\chi,s}$ is irreducible.
		\item $\sigma$ is irreducible for any $\alpha$ and $\sigma\in BZ_{L_{\alpha}}\bk{\pi}$.
	\end{itemize}
\end{Cor}
For a Levi subgroup, $L_{\alpha}$, of semi-simple rank 1 and an element $w\in W^{M,L_{\alpha}}$ the Levi factor $L_{\alpha}'$ of $L_{\alpha}$ is either $T$ or $L_{\alpha}$.
The Levi factor $M'$ of $M$ is either $T$ or $w^{-1}\bk{L_{\alpha}}$.
More precisely
\begin{itemize}
	\item If $w^{-1}\bk{\alpha_i}\in \Phi_M$ then $L_{\alpha}'=L_{\alpha}$ and $M'=w^{-1}\bk{L_{\alpha}}$.
	\item If $w^{-1}\bk{\alpha_i}\notin \Phi_M$ then $L_{\alpha}'=T$ and $M'=T$.
\end{itemize}


We recall that, for a regular $\mu\in \bfX\bk{T}$, $i_T^{L_\alpha}\mu$ is reducible if and only if $\gen{\mu,\check{\alpha}}=\pm 1$ since $L_{\alpha}$ has semi-simple rank $1$.

\begin{framed}
	\textbf{Calculating the regular points of reducibility:}
	
	First, given $\alpha\in\Delta$ and $\chi$ of order $m$, we calculate all points $s\in\R$ where $BZ_{L_{\alpha}}\bk{\pi}$ is reducible, where $\pi=i_M^G\Omega_{M,\chi,s}$.
	By the above discussion, it is given by:
	\[
	\set{s \in \R \mvert \exists \: \alpha \in \Phi_{G}^{+} \setminus \Phi_{M}^{+} : \  \gen{\chi_{M,s},\check{\alpha}} = \pm1}
	\]
%

	The list of reducible regular degenerate principal series is given by the intersection between this set and the set of $s\in\R$ such that $I_P\bk{\chi,s}$ is regular.
\end{framed}


We now wish to determine whether a non-regular $I_P\bk{\chi,s}$ is reducible.

\subsection{Part \rom{2} - Reducibility Tests}

We consider the following reducibility criterion (\cite[Lemma 3.1]{MR1658535}).
\begin{Lem}[\textbf{RC}]
	\label{Lem_RC}
	Let $\pi=i_M^G\Omega$.
	Assume there exist smooth representations $\Pi$ and $\sigma$ of $G$ of finite length and a Levi subgroup $L$ of $G$ such that:
	\begin{enumerate}
		\item $\pi\leq \Pi$, $\sigma\leq \Pi$.
		\item $r_L^G \pi +r_L^G \sigma \not\leq r_L^G\Pi$.
		\item $r_L^G\pi\not\leq r_L^G\sigma$.
	\end{enumerate}
	Then $\pi$ is reducible and admits a common irreducible subquotient with $\sigma$.
\end{Lem}

We also have:
\begin{Lem}
	\label{Lem_mult_of_antidominant_exponent}
	For $\lambda_{a.d.}\in W\cdot\chi_{M,s}$, such that $\Real\bk{\lambda_{a.d.}}$ is anti-dominant, it holds that
	\[
	mult_{r_T^G\pi}\bk{\lambda_{a.d.}} = \# \Stab_W\bk{\lambda_{a.d.}} .
	\]
\end{Lem}

\begin{proof}
	We first recall an algorithm to find all the reduced Weyl words $w\in W$ such that, for a given $\lambda\in \mathfrak{a}_{T,\R}^\ast$, $w\cdot \lambda$ is in the closed fundamental Weyl chamber.
	Write $\lambda=\suml_{i=1}^n a_i \omega_i$.
	Choose a coordinate $i$ such that $a_i<0$ and apply the simple reflection $s_i$.
	Note that, in this case, $\lambda\prec s_i\bk{\lambda}$, namely $\gen{s_i\cdot\lambda-\lambda,\check{\alpha}}\geq 0$ for any $\alpha\in\Delta$.
	Repeating this process will produce a reduced Weyl word $w=w[i_l,...,i_1]$ such that $w\cdot\lambda$ is dominant.
	Note that if $a_i<0$, then $s_i\bk{\lambda}\prec \lambda$.
	
	A similar approach can be used to find all Weyl elements $w\in W$ so that $w\cdot\lambda$ is in the negative Weyl chamber (i.e. anti-dominant).
		
	We say that a Weyl element $w\in W$ \emph{starts} with $s_i$ if $len\bk{ws_i}<len\bk{w}$.

	Write $\Real\bk{\chi_{M,s}}=\suml_{i=1}^{n} a_i\omega_i$.
	As $P=P_i$ we have $a_j=-1$ for all $i\neq j$ and $a_i\geq 0$.
	By the above mentioned algorithm, any Weyl word $w\in W$ such that $w\cdot \Real\bk{\chi_{M,s}}$ is anti-dominant,starts with $s_i$, as $a_i>0$ and does not start with $s_j$ for $j\neq i$ since $a_j<0$.
	On the other hand,
	\[
	W^{M,T}= \set{w\in W\mvert l\bk{ws_j}>l\bk{w}\ \forall j\neq i} .
	\]
	It follows that
	\[
	\set{w\in W\mvert w\cdot\Real\bk{\chi_{M,s}} \text{ anti-dominant}} \subset W^{M,T}.
	\]
\end{proof}



\begin{Remark}
	Given two degenerate principal series $i_M^G\Omega$ and $i_{M'}^G\Omega'$ such that $W\cdot r_M^G\Omega\cap W\cdot r_{M'}^G\Omega'\neq \emptyset$ it follows that $W\cdot r_M^G\Omega = W \cdot r_{M'}^G\Omega'$.
	Furthermore, \Cref{Lem_mult_of_antidominant_exponent} implies that for any anti-dominant $\lambda_{a.d.}\in r_T^Gi_M^G\Omega$ it holds that
	\[
	mult_{r_T^Gi_M^G\Omega}\bk{\lambda_{a.d.}} = mult_{r_T^Gi_{M'}^G\Omega'}\bk{\lambda_{a.d.}} = \# \Stab_W\bk{\lambda_{a.d.}} .
	\]
	In particular, assumption (2) in \Cref{Lem_RC} holds automatically for $\pi=i_M^G\Omega$, $\sigma=i_{M'}^G\Omega'$ and $\Pi=i_T^G\bk{r_T^M\Omega}$.
	Namely, $\pi$ and $\sigma$ share a common (spherical) subquotient.
\end{Remark}

For a given non-regular representation $I_P\bk{\chi,s}$, let $v_{P,\chi,s}$ denote the set of all $w\cdot \chi_{M,s}$, for $w\in W$ such that $\Real\bk{w\cdot \chi_{M,s}}$ is anti-dominant.
We call the elements of $v_{P,\chi,s}$ the \emph{anti-dominant exponents of $I_P\bk{\chi,s}$}.

\begin{framed}
	\textbf{Implementing the reducibility criterion for non-regular points:}
	
	We implement the reducibility test in two steps:
	\begin{enumerate}
		\item 
		
		Consider two triples $\bk{P,\chi.s}$ and $\bk{P',\chi',s}$ such that $I_P\bk{\chi,s}$ and $I_{P'}\bk{\chi',s'}$ share an anti-dominant exponent $\lambda_{a.d.}$.
		In order to show that $\pi$ is reducible, it will be enough to show that $\pi=I_P\bk{\chi,s}$, $\sigma=I_{P'}\bk{\chi',s'}$ and $\Pi=\Ind_B^G\chi_{M,s}$ satisfy the conditions in (\textbf{RC}).
		Conditions (1) and (2) automatically follow from \Cref{Eq_I_P_injects_to_I_B} and \Cref{Lem_mult_of_antidominant_exponent}.
		Condition (3) can be verified using the following methods:
		\begin{itemize}
			\item If $\#W\bk{G,P}>\#W\bk{G,P'}$, (3) holds.
			\item If $mult_{\pi}\bk{\chi_{M,s}}>mult_{\sigma}\bk{\chi_{M,s}}$, (3) holds.
			\item If $\suml_{W^{M,T}} w\circ\chi_{M,s} \not\leq \suml_{W^{M',T}} w\circ\chi'_{M',s'}$, (3) holds.
		\end{itemize}

		We note here that these tests could be performed to check the reducibility of $I_P\bk{\chi,s}$ and $I_{P'}\bk{\chi',s'}$ simultaneously.
		Also, we have performed these test at this order, as each test requires more computations than the previous one.
		
		\item For triples $\bk{P,\chi.s}$ such that reducibility could not be verified using the previous comparisons and irreducibility could not be verified either (see \Cref{Subsec:Irr_Test}), we performed the following reducibility test.
		
		Given $\lambda\in W\cdot\chi_{M,s}$ and
		\[
		\Theta = \set{\alpha\in\Delta \mvert \gen{\lambda,\check{\alpha}} = -1},
		\]
		$\lambda=r_T^{M_\Theta} \Omega'$, with $\Omega'\in\bfX\bk{M_{\Theta}}$.
		If the conditions of (\textbf{RC}) hold with respect to $\pi=I_P\bk{\chi,s}$, $\sigma=i_{M_\Theta}^G\Omega'$ and $\Pi=\Ind_B^G\chi_{M,s}$, then it follows that $\pi$ is reducible.
	
		We note here, that usually, in this case, there is a maximal Levi subgroup $M'$ of $G$ such that $i_{M_\Theta}^{M'}$ is an irreducible degenerate principal series of $M'$.
	\end{enumerate}

\end{framed}

\begin{Remark}
	For computational reasons, we applied step (2) only after the irreducibility test introduced in \Cref{Subsec:Irr_Test}.
\end{Remark}


\subsection{Part \rom{3} - Irreducibility Test}
\label{Subsec:Irr_Test}
%


After applying the reducibility test (1) explained above, one is left with a list of triples $\bk{P,\chi,s}$ such that $I_P\bk{\chi,s}$ is expected to be irreducible.
We now explain the method we used to test irreducibility of such representations.

The main idea we use to show that $\pi=I_P\bk{\chi,s}$ is irreducible goes as follows:
\begin{itemize}
	\item Let $v\leq r_T^G\pi$ be an exponent such that $\Real\bk{v}$ is anti-dominant and let $\pi'\leq \pi$ be an irreducible representation such that $v\leq r_T^G\pi'$.
	\item The Jacquet functor is exact and hence $r_T^G\pi'\leq r_T^G\pi$.
	\item If one can show that $r_T^G\pi'=r_T^G\pi$ (e.g. by showing that $\dim_\C r_T^G\pi'\geq \# W\bk{G,P} = \dim_\C r_T^G\pi$), then $\pi'=\pi$ is irreducible.
\end{itemize}

In order to prove that $r_T^G\pi'=r_T^G\pi$ we use \emph{branching rules}.
\\

\paragraph{\textbf{Branching Rules:}}
Let $L$ be a Levi subgroup of $G$ and $\lambda\leq r_T^G\pi'$ such that there exists a \underline{unique} irreducible representation $\sigma$ of $L$ such that $\lambda\leq r_T^L\sigma$.
Then the following hold:
\begin{itemize}
	\item The multiplicity $mult_{r_T^L\sigma}\bk{\lambda}$ divides the multiplicity $mult_{r_T^G\pi'}\bk{\lambda}$.
	\item If $n=\frac{mult_{r_T^G\pi'}\bk{\lambda}}{mult_{r_T^L\sigma}\bk{\lambda}}$ then $n\times \coset{r_T^L\sigma}\leq \coset{r_T^G\pi'}$.
\end{itemize}

\begin{proof}
	Indeed, write
	\[
	\coset{r_L^G\pi'} = \sum_{i=1}^{k} n_i\times\coset{\sigma_i},
	\]
	where $\sigma_i$ are disjoint irreducible representations of $L$.
	Since $r_T^G=r_T^Lr_L^G$ it follows that
	\[
	\coset{\lambda}\leq \coset{r_T^G\pi'} = \sum_{i=1}^{k} n_i \times \coset{r_T^L\sigma_i},
	\]
	Without any loss of generality, $\lambda\leq r_T^L\sigma_1$.
	Since $\sigma$ is the unique irreducible representation of $L$ such that $\lambda\leq r_T^L\sigma$ then:
	\begin{itemize}
		\item $\sigma_1=\sigma$.
		\item $\lambda\not\leq r_T^L\sigma_i$ for $i>1$.
		\item $mult_{r_T^G\pi'}\bk{\lambda}=n_1\cdot mult_{r_T^L\sigma}\bk{\lambda}$.
		\item $n_1\times\coset{r_T^L\sigma} \leq r_T^G\pi'$.
	\end{itemize}
\end{proof}

A list of branching rules, which were used by us in \Cref{Sec_degenerate_principal_series_E6}, can be found in \Cref{App:Branching_Rules}.
These branching rules are associated with Levi subgroups $M$ such that $M^{der}$ is of type $A_1$, $A_2$, $A_3$ or $D_4$. \\

\begin{framed}
	\textbf{Implementing the irreducibility criterion for non-regular points:}
	
	We start by choosing an anti-dominant exponent of $I_P\bk{\chi,s}$.
	Namely, we choose an anti-dominant $\lambda_{a.d.}$ in $r_T^G\bk{I_P\bk{\chi,s}}$.
	Also, let $\pi'\leq \pi$ be an irreducible subquotient such that $\lambda_{a.d.}\leq r_T^G\pi'$.
	It follows from \Cref{Eq:Golden_rule} that $\pi'$ is the unique subquotient of $\pi$ with that property and that $\Card{\Stab_{W}\bk{\lambda_{a.d.}}}\times\lambda_{a.d.}\leq r_T^G\pi'$.
	
	We then create a list of pairs $\coset{\lambda,n_\lambda}$, where $\lambda$ is an exponent of $I_P\bk{\chi,s}$ and $n_\lambda$ is a lower bound on the multiplicity of $\lambda$ in $\pi'$.
	For any $\lambda$ in the list we check which of the branching rules (see \Cref{App:Branching_Rules}) can be applied to $\lambda$ and update the bounds $n_\lambda$ accordingly.
	The process is terminated when either
	\[
	\sum_\lambda n_\lambda\times\coset{\lambda} =r_T^G\pi,
	\]
	in which case $\pi=\pi'$ is irreducible, or that no more branching rules can be further applied, in which case the algorithm is inconclusive.
\end{framed}

\begin{Remark}
	In \Cref{App:Example_of_Irr_Cir} we give an example of a proof of irreducibility using this algorithm.
\end{Remark}



\subsection{Part \rom{4} - Irreducible Quotients and Subrepresentations}
\label{Subsec:Subs}
We now describe the tools used to calculate the maximal semi-simple subrepresentation of $\pi=i_M^G \Omega$ and its maximal semi-simple quotient.

First, note that, by contragredience, it is enough to find its maximal semi-simple subrepresentation.
Furthermore, note the following bound on its length:
\begin{Lem}
	\label{Lem:Bound_on_length_of_ss_subrepn}
	For  $\Omega=\Omega_{M,\chi,s}$ and $\pi=i_M^G\bk{\Omega}$, the number of irreducible subrepresentations of $\pi$ is bounded by $mult_\pi\bk{\chi_{M,s}}$.
\end{Lem}


\begin{proof}
	We prove that for any irreducible $\pi'$ such that
	\[
	\pi'\hookrightarrow\pi = i_M^G\Omega
	\]
	it holds that $\chi_{M,s}\leq r_T^G\pi'$.
	Note that
	\[
	\pi'\hookrightarrow i_M^G\Omega \hookrightarrow i_T^G \chi_{M,s}.
	\]
	Hence, $\pi'$ is a subrepresentation of $i_T^G \chi_{M,s}$.
	
	On the hand, by Frobenius reciprocity,
	\[
	\Hom_G\bk{\pi',i_T^G \chi_{M,s}} \cong \Hom_T\bk{r_T^G\pi',\chi_{M,s}} .
	\]
	Since the left-hand side is non-trivial, so is the right-hand side.
	It follows that $\chi_{M,s}\leq r_T^G\pi'$.
	The claim then follows.
	
\end{proof}

\begin{Cor}
	\label{Lemma:Quotients_of_Regular_Induction}
	If $\Omega=\Omega_{M,\chi,s}$ is regular then $\pi=i_M^G\bk{\Omega}$ admits a unique irreducible quotient and a unique irreducible subrepresentation.
	Each of them has multiplicity one in $\pi$.
\end{Cor}

We now describe the various methods used by us for calculating the maximal semi-simple subrepresentation of $\pi$. \\

\underline{\textbf{Case I:}}
If $mult_\pi\bk{\chi_{M,s}}=1$, then $\pi=i_M^G\bk{\Omega}$ admits a unique irreducible subrepresentation.
This subrepresentation, $\pi_0$, appears with multiplicity $1$ in $\pi$ as $mult_\pi\bk{\chi_{M,s}} = mult_{\pi_0}\bk{\chi_{M,s}}$.

\begin{framed}
	\textbf{Implementing Case I:}
	This only requires calculating $mult_\pi\bk{\chi_{M,s}}$ and show that it is $1$.
	If that is not the case, we move to the next cases.
\end{framed}

\begin{Remark}
	We note here that in all of the relevant cases, if $s>0$, then $mult_\pi\bk{\chi_{M,s}}=1$ so $I_P\bk{\chi,s}$ admits a unique irreducible subrepresentation.
	This, in fact, follows directly from \cite[Theorem 6.3]{MR2490651}.
	
	It is thus enough, to consider $s\leq 0$.
\end{Remark}

\underline{\textbf{Case II:}}
Assume that $m=mult_{\pi}\bk{\chi_{M,s}}$ and that, using a sequence of branching rules, one can show that there exists an irreducible representation $\pi_0\leq\pi$ of $G$ such that $m\times\coset{\chi_{M,s}} \leq \coset{r_T^G\pi_0}$.
Then, $\pi=i_M^G\bk{\Omega}$ admits a unique irreducible subrepresentation.
This subrepresentation, $\pi_0$, appears with multiplicity $1$ in $\pi$ as $mult_\pi\bk{\chi_{M,s}} = mult_{\pi_0}\bk{\chi_{M,s}}$.

\begin{framed}
	\textbf{Implementing Case II:}
	This calculation follows the ideas of \Cref{Subsec:Irr_Test}.
	Here we do not aim to show that $\sum_\lambda n_\lambda\times\coset{\lambda} =\coset{r_T^G\pi}$; rather we aim to show that $n_{\chi_{M,s}}=mult_\pi\bk{\chi_{M,s}}$.
	
	As in the irreducibility criterion, we start with an anti-dominant exponent $\lambda_{a.d.}\leq r_T^G\pi$.
	Since we assume that $s\leq 0$, there is a subrepresentation $\pi_0$ of $\pi$ such that $\lambda_{a.d.}\leq r_T^G \pi_0$.
	
	We then start to apply branching rules on $\lambda_{a.d.}$.	
	This computation would terminate if it found that $n_{\chi_{M,s}}=mult_\pi\bk{\chi_{M,s}}$, in which case $\pi_0$ is the unique irreducible subrepresentation of $\pi$, or if no more branching rules can be further applied, in which case this test is inconclusive.
\end{framed}

%
%
%
%
%

We now describe one more possible case, which occurred in the course of this work.
Due to the scarcity of its use, we did not implement it as part of our algorithm.
The triples $\bk{M,s,\chi}$ in which these cases were relevant will be further discussed in \Cref{Sec_degenerate_principal_series_E6}.


\underline{\textbf{Case III:}}
Let:
\begin{itemize}
	\item $M'$ be a Levi subgroup of $G$.
	\item $M'' = M \cap M'$.
	\item $\Omega''=r_{M''}^M\Omega$
	\item Assume that $i_{M''}^{M'}\Omega'' = \oplus \sigma_i$ is semi-simple.
\end{itemize}

By induction in stage, it follows that
\begin{align*}
i_M^G\Omega 
&\hookrightarrow i_{M}^G \bk{i_{M''}^M \Omega''} \\
&\cong i_{M''}^G\Omega'' \\
&\cong i_{M'}^G \bk{i_{M''}^{M'}\Omega''} = \oplus i_{M'}^G \sigma_i
\end{align*}

Furthermore, assume that each of the $i_{M'}^G \sigma_i$ admits a unique irreducible subrepresentation $\pi_i$.
It follows that the maximal semi-simple subrepresentation of $\pi$ is a subrepresentation of $\oplus \pi_i$.
It remains to determine which of the $\pi_i$-s is a subrepresentation of $\pi$ and the equivalencies between them.
Both of these can often be done by a comparison of Jacquet modules and multiplicities of certain exponents in them.

We finish this section by noting a useful fact, which other sources sometimes refer to as a \textbf{central character argument}.
This is, in a sense, a counterpart of \Cref{Lem:Bound_on_length_of_ss_subrepn}.

\begin{Lem}
	\label{Central_character_argument}
	Let $\pi'$ be a smooth irreducible representation of $G$.
	If $\lambda\leq r_T^G\pi'$ then $\pi'\hookrightarrow i_T^G\lambda$.
\end{Lem}

\begin{proof}
	Let $\rho=r_T^G\pi'$ and for any $\omega:T\to\C^\times$ let
	\[
	V_{\omega,n} = \set{v\in V \mvert \bk{\rho\bk{t}-\omega\bk{t}}^nv = 0\ \forall t\in T},
	\]
	where $V$ denotes the underlying (finite dimensional) vector space of $\rho$.
	Furthermore, let
	\[
	V_{\omega,\infty} = \cup_{n\geq 1} V_{\omega,n},\quad V_\omega=V_{\omega,1} .
	\]
	From \cite[Proposition 2.1.9]{Casselman} (note that $Z_T=T$) it follows that
	\begin{equation}
	\label{Eq:decomposition_along_central_characters}
	V=\bigoplus_{\omega:T\to\C^\times} V_{\omega,\infty},
	\end{equation}
	where $V_{\omega,\infty}\neq \set{0}$ for only finitely many $\omega$.
	Namely,
	\[
	\rho = \bigoplus_{i\in I} V_{\omega_i,\infty},
	\]
	for a finite set $I$.
	
	By Frobenius reciprocity,
	\[
	\Hom_{G}\bk{\pi',i_T^G\lambda} \cong \Hom_{T}\bk{r_T^G\pi',\lambda} = \Hom_{T}\bk{\rho,\lambda}
	\]	
	Since $\lambda\leq \rho$, it follows that $V_{\lambda}\neq\set{0}$.
	By \Cref{Eq:decomposition_along_central_characters},
	\[
	\Hom_{T}\bk{r_T^G\pi,\lambda} \cong \bigoplus_{i\in I} \Hom_{T}\bk{\rho_{i,\infty},\lambda }\cong  \Hom_{T}\bk{\rho_{\lambda,\infty},\lambda}\neq\set{0} .
	\]
	
	We conclude that $\Hom_{G}\bk{\pi,i_T^G\lambda}\neq \set{0}$ and hence
	\[
	\pi\hookrightarrow i_T^G\lambda .
	\]	

\end{proof}

\section{Degenerate Principal Series Representations of $E_6$}
\label{Sec_degenerate_principal_series_E6}

Let $G$ be the simple, split and simply-connected $F$-group of type $E_6$.
In this section, we describe the structure of the degenerate principal series of $G$ using the algorithm described in \Cref{Sec_algorithm}.
Namely, we determine the following:
\begin{itemize}
	\item All regular reducible degenerate principal series $I_P\bk{\chi,s}$ of $G$.
	\item All non-regular degenerate principal series $I_P\bk{\chi,s}$ of $G$.
	\item For each non-regular $I_P\bk{\chi,s}$, we determine whether it is reducible or irreducible.
	\item For reducible $I_P\bk{\chi,s}$ we determine its maximal semi-simple subrepresentation and quotient.
	In fact, we show that, with only one exception, all $I_P\bk{\chi,s}$ admit a unique irreducible quotient and a unique irreducible subrepresentation.
\end{itemize}

\subsection{The Exceptional Group of Type $E_6$}

We start by describing the structure of $G$.
We consider $G$ as a simply-connected Chevalley group of type $E_6$ (see \cite[pg. 21]{MR0466335}).
Namely, we fix a Borel subgroup $B$ whose Levi subgroup is a torus $T$.
This gives rise to a set of 72 roots $\Phi$, containing a set of simple roots $\Delta=\set{\alpha_1,\alpha_2,\alpha_3,\alpha_4,\alpha_5,\alpha_6}$.
The Dynkin diagram of type $E_6$ is
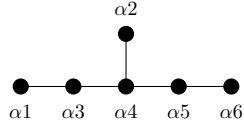
\begin{center}
	\begin{tikzpicture}
	\dynkin[scale=2]{E}{6}
	\rootlabel{1}{\alpha 1}
	\rootlabel{2}{\alpha 2}
	\rootlabel{3}{\alpha 3}
	\rootlabel{4}{\alpha 4}
	\rootlabel{5}{\alpha 5}
	\rootlabel{6}{\alpha 6}
	\end{tikzpicture}
\end{center}
The group $G$ is generated by the symbols
\[
\set{x_\alpha\bk{r} \mvert \alpha\in\Phi \ r\in F}
\]
subject to the Chevalley-Steinberg relations (see \cite[pg. 66]{MR0466335}).

We record the isomorphism classes of standard Levi subgroups $M_\Theta$ of $G$, with $\Theta$ inducing a connected sub-Dynkin diagram, in the following lemma.
\begin{Lem}
	Let $\Theta\subset\Delta$ induce a connected sub-diagram.
	Then, $\Theta$ is either $A_n$ (with $1\leq n\leq 5$), $D_4$ or $D_5$.
	Furthermore, $M_\Theta$ is classified as follows:
	\begin{itemize}
		\item If $\Theta$ is of type $A_n$, with $n\leq 4$, then $M_\Theta \cong GL_{n+1}\times GL_1^{5-n}$.
		
		\item If $\Theta$ is of type $D_5$, then $M_\Theta\cong GSpin_{10}$.
		
		\item If $\Theta$ is of type $D_4$, then $M_\Theta\cong GL_1\times GSpin_{8}$.
		
		\item If $\Theta$ is of type $A_5$, then $M_\Theta\cong \mu_2\lmod \bk{SL_6\times GL_1}$ is an algebraic triple cover of $GL_6$.
	\end{itemize}
\todonum{Maybe add $M_i$ for $i=3,4,5$? Also for other $M_i$, state explicitly that this is the Levi subgroup treated.}
\end{Lem}

\begin{proof}
	If $\Theta$ is of type $D_5$, then this is the Levi subgroup of the parabolic subgroup of $G$ whose nilradical is Abelian considered in \cite{MR1993361}.
	If $\Theta$ is of type $D_4$ or $A_n$ with $n\leq 4$, this follows from \cite{MR1913914} together with the fact that $\Theta$ is contained in a sub-Dynkin diagram of type $D_5$.
	If $\Theta$ is of type $A_5$ this follows from a direct calculation.
\end{proof}

Let $P_i$ denote the maximal standard parabolic subgroup associated to $\Theta_i=\Delta\setminus\set{\alpha_i}$.
Let $M_i=M_{\alpha_i}=M_{\Theta_i}$ denote the Levi subgroup of $P_i$.

Let $W$ denote the Weyl group of $G$ associated with $T$.
It holds that $\Card{W} = 51,840$.
The Weyl group $W_{M_i}$ of the maximal Levi subgroup $M_i$ has the following cardinality
\[
\Card{W_{M_i}} = \piece{
	1,920,& \text{if } i=1,6 \\
	720,&  \text{if } i=2 \\
	240,& \text{if } i=3,5 \\
	72,& \text{if } i=4 } .
\]
Hence, the set $W^{M_i,T}$ of minimal representatives of $W\rmod W_{M_i}$ has the following cardinality
\[
\Card{W^{M_i,T}} = \Card{W\rmod W_{M_i}}
= \frac{\Card{W}}{\Card{W_{M_i}}} 
= \piece{
	27,& \text{if } i=1,6 \\
	72,&  \text{if } i=2 \\
	216,& \text{if } i=3,5 \\
	720,& \text{if } i=4 }.
\]
We note here that for a 1-dimensional representation $\Omega$ of $M_i$, $r_T^Gi_{M_i}^G\Omega$ has dimension $\Card{W^{M_i,T}}$.

We also note here that any $\lambda\in \bfX\bk{T}$ can be written as the following combination
\[
\suml_{i=1}^6 \Omega_i\circ\omega_{\alpha_i}.
\]
As a shorthand, we will write
\begin{equation}
\esixchar{\Omega_1}{\Omega_2}{\Omega_3}{\Omega_4}{\Omega_5}{\Omega_6}
=\suml_{i=1}^6 \Omega_i\circ\omega_{\alpha_i}.
\end{equation}
Also, let $\esixcharclass{\Omega_1}{\Omega_2}{\Omega_3}{\Omega_4}{\Omega_5}{\Omega_6}$ denote the class in $\mathfrak{R}\bk{T}$ of $\esixchar{\Omega_1}{\Omega_2}{\Omega_3}{\Omega_4}{\Omega_5}{\Omega_6}$.

\subsection{The Degenerate Principal Series of $E_6$}

Before stating the results, we make a few general comments.
\begin{itemize}
	\item By contragredience, it is enough to state the results for $s\leq 0$.
	
	\item For short, we write $\coset{i,s,k}$ for $I_P\bk{\chi,s}$ with $P=P_i$ and $\chi$ of order $k$.
	
	\item By the action of the outer automorphism group ,$\Z\rmod2\Z$ of $Dyn\bk{E_6}$, it is enough to consider the parabolic subgroups $P_1$, $P_2$, $P_3$ and $P_4$.
	The results for $P_5$ and $P_6$ are similar to those of $P_1$ and $P_3$ respectively.
\end{itemize}

\begin{Thm}
	\label{Thm:Deg_PS_of_E6}
	For any $1\leq i \leq 4$, all reducible regular $I_P\bk{\chi,s}$ and all non-regular $I_P\bk{\chi,s}$ are given in the following tables.

	
	\begin{itemize}
		\item For $P=P_1$ and $P=P_6$:

		\begin{table}[H]
			\begin{tabular}{|c|c|c|c|c|c|c|c|}
				\hline
				\diagbox{$ord\bk{\chi}$}{$s$} & $-6$ & $-5$ & $-4$ & $-3$ & $-2$ & $-1$ & $0$ \\ \hline
				$1$ & \makecell{reg. \\ red.} & \makecell{non-reg. \\ irr.} & \makecell{non-reg. \\ irr.} & \makecell{non-reg. \\ red.} & \makecell{non-reg. \\ irr.} & \makecell{non-reg. \\ irr.} & \makecell{non-reg. \\ irr.} \\ \hline
			\end{tabular}
			\label{Table_Results_for_P_1}
		\end{table}
	
		\item For $P=P_2$:

		\begin{table}[H]
			\begin{tabular}{|c|c|c|c|c|c|c|c|}
				\hline
				\diagbox{$ord\bk{\chi}$}{$s$} & $-\frac{11}{2}$ & $-\frac{9}{2}$ & $-\frac{7}{2}$ & $-\frac{5}{2}$ & $-\frac{3}{2}$ & $-\frac{1}{2}$ & $0$ \\ \hline
				$1$ & \makecell{reg. \\ red.} & \makecell{non-reg. \\ irr.} & \makecell{non-reg. \\ red.} & \makecell{non-reg. \\ red.} & \makecell{non-reg. \\ irr.} & \makecell{non-reg. \\ red.} & \makecell{non-reg. \\ irr.} \\ \hline
				$2$ & \makecell{reg. \\ irr.} & \makecell{reg. \\ irr.} & \makecell{reg. \\ irr.} & \makecell{reg. \\ irr.} & \makecell{reg. \\ irr.} & \makecell{reg. \\ red.} & \makecell{non-reg. \\ irr.} \\ \hline
			\end{tabular}
			\label{Table_Results_for_P_2}
		\end{table}
	
		\item For $P=P_3$ or $P=P_5$:

		\begin{table}[H]
			\begin{tabular}{|c|c|c|c|c|c|c|c|}
				\hline
				\diagbox{$ord\bk{\chi}$}{$s$} & $-\frac{9}{2}$ & $-\frac{7}{2}$ & $-\frac{5}{2}$ & $-\frac{3}{2}$ & $-1$ & $-\frac{1}{2}$ & $0$ \\ \hline
				$1$ & \makecell{reg. \\ red.} & \makecell{non-reg. \\ red.} & \makecell{non-reg. \\ red.} & \makecell{non-reg. \\ red.} & \makecell{non-reg. \\ irr.} & \makecell{non-reg. \\ irr.} & \makecell{non-reg. \\ irr.} \\ \hline
				$2$ & \makecell{reg. \\ irr.} & \makecell{reg. \\ irr.} & \makecell{reg. \\ irr.} & \makecell{reg. \\ red.} & \makecell{non-reg. \\ irr.} & \makecell{non-reg. \\ irr.} & \makecell{non-reg. \\ irr.} \\ \hline
			\end{tabular}
			\label{Table_Results_for_P_3}
		\end{table}
	
		\item For $P=P_4$:

		\begin{table}[H]
			\begin{tabular}{|c|c|c|c|c|c|c|c|}
				\hline
				\diagbox{$ord\bk{\chi}$}{$s$} & $-\frac{7}{2}$ & $-\frac{5}{2}$ & $-\frac{3}{2}$ & $-1$ & $-\frac{1}{2}$ & $-\frac{1}{6}$ & $0$ \\ \hline
				$1$ & \makecell{reg. \\ red.} & \makecell{non-reg. \\ red.} & \makecell{non-reg. \\ red.} & \makecell{non-reg. \\ red.} & \makecell{non-reg. \\ red.} & \makecell{non-reg. \\ irr.} & \makecell{non-reg. \\ irr.} \\ \hline
				$2$ & \makecell{reg. \\ irr.} & \makecell{reg. \\ irr.} & \makecell{reg. \\ red.} & \makecell{non-reg. \\ red.} & \makecell{non-reg. \\ red.} & \makecell{reg. \\ irr.} & \makecell{non-reg. \\ irr.} \\ \hline
				$3$ & \makecell{reg. \\ irr.} & \makecell{reg. \\ irr.} & \makecell{reg. \\ irr.} & \makecell{reg. \\ irr.} & \makecell{non-reg. \\ red.} & \makecell{non-reg. \\ irr.} & \makecell{reg. \\ irr.} \\ \hline
			\end{tabular}
			\label{Table_Results_for_P_4}
		\end{table}

	\end{itemize}

	All of the above representations admit a unique irreducible subrepresentation and a unique irreducible quotient, with the exception of $\coset{4,-\frac{1}{2},3}$, in which case $I_P\bk{\chi,s}$ admits a unique irreducible quotient and a maximal semi-simple subrepresentation of the form $\sigma_1\oplus \sigma_2\oplus\sigma_3$ (where the $\sigma_i$ are inequivalent).
	
	Furthermore, for all cases, any irreducible subrepresentation or quotient of $I_P\bk{\chi,s}$ appears in $I_P\bk{\chi,s}$ with multiplicity $1$.
	
\end{Thm}

\begin{Remark}
	We note here that our results for the parabolic subgroup $P_1$ agree with the results of \cite{MR1993361} (see page 297).
	In particular, the unique irreducible subrepresentation, $\Pi_{min.}$, of $I_{P_1}\bk{\Id,-3}$ is the minimal representation of $G$.
	By examining the anti-dominant exponent of $\Pi_{min.}$, one sees that this is also the unique irreducible subrepresentation of $I_{P_2}\bk{\Id,-\frac{7}{2}}$, $I_{P_6}\bk{\Id,-3}$ and $I_{P_4}\bk{\Id,-\frac{5}{2}}$.
	This fact gives rise to a Siegel-Weil identity between the relevant Eisenstein series at these points (see \cite[pg. 68]{HeziMScThesis}).
\end{Remark}

\begin{proof}
	
	We deal with the question of reducibility separately from the calculation of maximal semi-simple subrepresentations. \\
	
	\textbf{Reducibility}
	
	For most $I_P\bk{\chi,s}$, the algorithm provided in \Cref{Sec_algorithm} suffices to yield the results stated in the theorem.
	In particular, reducibility of regular $I_P\bk{\chi,s}$ and irreducibility of non-regular $I_P\bk{\chi,s}$ is completely determined by the tools described there.
	The branching rules used in this proof are listed in \Cref{App:Branching_Rules}.
	See \Cref{App:Example_of_Irr_Cir} for an example of a proof of the irreducibility of $I_P\bk{\chi,s}$, in the case $\coset{1,-2,1}$, using branching rules.
	
	We note that in the case $\coset{3,-\frac{3}{2},1}$, the branching rules supplied in \Cref{App:Branching_Rules} did not cover all of $r_T^G\pi$.
	One can still prove irreducibility as will be explained later.
	
	For most non-regular reducible $I_P\bk{\chi,s}$, reducibility can be determined by (\textbf{RC}) (\Cref{Lem_RC}) by setting $\Pi=i_T^G\chi_{M,s}$ and $\sigma$ being another degenerate principal series.
	In such a case, we quote one triple $\coset{i,s,k}$ relevant for the proof in the following table.
	Usually, there is more than one such triple, in which case we recorded only one.
	
	A few cases required the application of (\textbf{RC}) with respect to an induction from a non-maximal parabolic $M_\Theta$ such that $\Card{\Theta}=4$.
	In such a case, we write $\coset{\coset{i_1,i_2},\coset{s_1,s_2},\coset{k_1,k_2}}$ to represent the induction from $M=M_\Theta$ with $\Theta=\Delta\setminus\set{\alpha_{i_1},\alpha_{i_2}}$ with initial exponent
	\[
	-\bk{\suml_{j\neq i_1,i_2}\omega_{\alpha_j}} + \bk{s_1+k_1\cdot\chi}\circ\omega_{\alpha_{i_1}} + \bk{s_2+k_2\cdot\chi}\circ\omega_{\alpha_{i_2}} ,
	\]
	where $\chi$ is of order $k$.
	In this case, again, there is more than one possible choice of data $\coset{\coset{i_1,i_2},\coset{s_1,s_2},\coset{k_1,k_2}}$ that will yield a proof of reducibility; we record only one such choice.
	
	The reducibility in the remaining case, $\coset{4,-\frac{1}{2},3}$, follows from the fact that, in that case, $I_P\bk{\chi,s}$ admits a maximal semi-simple subrepresentation of length $3$ as will be shown later.

	\begin{itemize}
		\item For $P=P_1$

		\begin{table}[H]
			\begin{tabular}{|c|c|c|c|c|c|c|c|}
				\hline
				\diagbox{$ord\bk{\chi}$}{$s$} & $-6$ & $-5$ & $-4$ & $-3$ & $-2$ & $-1$ & $0$ \\ \hline
				$1$ & \notableentry & \notableentry & \notableentry & $\coset{6,-3,1}$ & \notableentry & \notableentry & \notableentry \\ \hline
			\end{tabular}
			\label{Table_Reducibility_proof_for_P_1}
		\end{table}
		
		\item For $P=P_2$

		\begin{table}[H]
			\begin{tabular}{|c|c|c|c|c|c|c|c|}
				\hline
				\diagbox{$ord\bk{\chi}$}{$s$} & $-\frac{11}{2}$ & $-\frac{9}{2}$ & $-\frac{7}{2}$ & $-\frac{5}{2}$ & $-\frac{3}{2}$ & $-\frac{1}{2}$ & $0$ \\ \hline
				$1$ & \notableentry & \notableentry & $\coset{1,-3,1}$ & $\coset{1,0,1}$ & \notableentry & $\coset{\coset{1,2},\coset{0,-1},\coset{0,0}}$ & \notableentry \\ \hline
			\end{tabular}
			\label{Table_Reducibility_proof_for_P_2}
		\end{table}
		
		\item For $P=P_3$

		\begin{table}[H]
			\begin{tabular}{|c|c|c|c|c|c|c|c|}
				\hline
				\diagbox{$ord\bk{\chi}$}{$s$} & $-\frac{9}{2}$ & $-\frac{7}{2}$ & $-\frac{5}{2}$ & $-\frac{3}{2}$ & $-1$ & $-\frac{1}{2}$ & $0$ \\ \hline
				$1$ & \notableentry & $\coset{1,-4,1}$ & $\coset{6,-1,1}$ & $\coset{2,-\frac{1}{2},1}$ & \notableentry & \notableentry & \notableentry \\ \hline
			\end{tabular}
			\label{Table_Reducibility_proof_for_P_3}
		\end{table}
		
		\item For $P=P_4$

		\begin{table}[H]
			\begin{tabular}{|c|c|c|c|c|c|c|c|}
				\hline
				\diagbox{$ord\bk{\chi}$}{$s$} & $-\frac{7}{2}$ & $-\frac{5}{2}$ & $-\frac{3}{2}$ & $-1$ & $-\frac{1}{2}$ & $-\frac{1}{6}$ & $0$ \\ \hline
				$1$ & \notableentry & $\coset{1,-3,1}$ & $\coset{2,-\frac{1}{2},1}$ & $\coset{3,0,1}$ & $\coset{\coset{3,6},\coset{\frac{1}{2},-\frac{3}{2}},\coset{0,0}}$ & \notableentry & \notableentry \\ \hline
				$2$ & \notableentry & \notableentry & \notableentry & $\coset{3,0,2}$ & $\coset{\coset{1,5},\coset{\frac{1}{2},\frac{1}{2}},\coset{1,1}}$ & \notableentry & \notableentry \\ \hline
				$3$ & \notableentry & \notableentry & \notableentry & \notableentry & $\ast$ & \notableentry & \notableentry \\ \hline
			\end{tabular}
			\label{Table_Reducibility_proof_for_P_4}
		\end{table}
		
	\end{itemize}

	It is left to explain how to prove the irreducibility of $\pi=I_P\bk{\chi,s}$ in the case $\coset{3,-\frac{1}{2},1}$.
	The anti-dominant exponent in this case is given by
	\[
	\lambda_{a.d.}=\esixchar{-1}{0}{0}{-1}{0}{0} .
	\]
	As mentioned above, applying the branching rules from \Cref{App:Branching_Rules} on exponents, starting with $\lambda_{a.d.}$, doesn't result with the full Jacquet module $r_T^G\pi$.
	Still, irreducibility follows from that calculation, as will be explained now.
	Let
	\[
	\lambda_0 = \esixchar{-1}{-1}{3}{-1}{-1}{-1}, \quad
	\lambda_1 = \esixchar{-1}{-1}{-1}{-1}{4}{-1},
	\]
	denote the leading exponents of $\pi$ and its contragredient $\widetilde{\pi}$.
	Since $mult\bk{\lambda_0,\pi}=mult\bk{\lambda_1,\pi}=1$ it follows that $\pi$ admits a unique irreducible subrepresentation and a unique irreducible quotient.
	Let $\pi_0\leq \pi$ be the irreducible subquotient of $\pi$ such that $\lambda_{a.d.}\leq r_T^G\pi_0$.
	Applying the branching rules in \Cref{App:Branching_Rules} on $\lambda_{a.d.}$ yields that $\lambda_0,\lambda_1\leq r_T^G\pi_0$.
	Hence, $\pi_0$ is both the unique irreducible subrepresentation and the unique irreducible quotient of $\pi$.
	It follows that $\pi=\pi_0$ is irreducible. \\

	\todonum{If this yields a new branching rule, helpful to calculations in $E_7$, then the new rule should be stated here.}

	\textbf{Subrepresentations}

	For all $I_P\bk{\chi,s}$ with $s> 0$ (regardless of regularity and reducibility), the fact that they admit a unique irreducible subrepresentation follows from \Cref{Lemma:Quotients_of_Regular_Induction}.
	By contragredience, $I_P\bk{\chi,s}$, with $s\leq 0$, admits a unique irreducible quotient.
	Using the same reasoning, we see that this quotient appears in $I_P\bk{\chi,s}$ with multiplicity $1$.
	
	We now treat the maximal semi-simple subrepresentation of reducible non-regular $I_P\bk{\chi,s}$, with $s\leq 0$.
	First of all, note that, by \Cref{Lemma:Quotients_of_Regular_Induction}, when $I_P\bk{\chi,s}$ is regular, it admits a unique irreducible subrepresentation.
	
	When $I_P\bk{\chi,s}$ is non-regular, the fact that $I_P\bk{\chi,s}$ admits a unique irreducible subrepresentation follows from the considerations explained in \Cref{Subsec:Subs}.
	In the following tables, we list all cases which follow from \textbf{Case I} and \textbf{Case II} or \textbf{Case III}.
	For points where we used \textbf{Case III} ($\coset{4,-\frac{1}{2},1}$ and $\coset{4,-\frac{1}{2},3}$), the argument is detailed at the end of this proof.
	
	\begin{itemize}
		\item For $P=P_1$

		\begin{table}[H]
			\begin{tabular}{|c|c|c|c|c|c|c|c|}
				\hline
				\diagbox{$ord\bk{\chi}$}{$s$} & $-6$ & $-5$ & $-4$ & $-3$ & $-2$ & $-1$ & $0$ \\ \hline
				$1$ & \notableentry & \notableentry & \notableentry & \textbf{Case I} & \notableentry & \notableentry & \notableentry \\ \hline
			\end{tabular}
			\label{Table_UIS_proof_for_P_1}
		\end{table}
		
		\item For $P=P_2$

		\begin{table}[H]
			\begin{tabular}{|c|c|c|c|c|c|c|c|}
				\hline
				\diagbox{$ord\bk{\chi}$}{$s$} & $-\frac{11}{2}$ & $-\frac{9}{2}$ & $-\frac{7}{2}$ & $-\frac{5}{2}$ & $-\frac{3}{2}$ & $-\frac{1}{2}$ & $0$ \\ \hline
				$1$ & \notableentry & \notableentry & \textbf{Case I} & \textbf{Case II} & \notableentry & \textbf{Case I} & \notableentry \\ \hline
			\end{tabular}
			\label{Table_UIS_proof_for_P_2}
		\end{table}
		
		\item For $P=P_3$

		\begin{table}[H]
			\begin{tabular}{|c|c|c|c|c|c|c|c|}
				\hline
				\diagbox{$ord\bk{\chi}$}{$s$} & $-\frac{9}{2}$ & $-\frac{7}{2}$ & $-\frac{5}{2}$ & $-\frac{3}{2}$ & $-1$ & $-\frac{1}{2}$ & $0$ \\ \hline
				$1$ & \notableentry & \textbf{Case II} & \textbf{Case II} & \textbf{Case II} & \notableentry & \notableentry & \notableentry \\ \hline
			\end{tabular}
			\label{Table_UIS_proof_for_P_3}
		\end{table}
		
		\item For $P=P_4$

		\begin{table}[H]
			\begin{tabular}{|c|c|c|c|c|c|c|c|}
				\hline
				\diagbox{$ord\bk{\chi}$}{$s$} & $-\frac{7}{2}$ & $-\frac{5}{2}$ & $-\frac{3}{2}$ & $-1$ & $-\frac{1}{2}$ & $-\frac{1}{6}$ & $0$ \\ \hline
				$1$ & \notableentry & \textbf{Case II} & \textbf{Case II} & \textbf{Case II} & \textbf{Case III} & \notableentry & \notableentry \\ \hline
				$2$ & \notableentry & \notableentry & \notableentry & \textbf{Case II} & \textbf{Case II} & \notableentry & \notableentry \\ \hline
				$3$ & \notableentry & \notableentry & \notableentry & \notableentry & \textbf{Case III} & \notableentry & \notableentry \\ \hline
			\end{tabular}
			\label{Table_UIS_proof_for_P_4}
		\end{table}
		
	\end{itemize}

	Also, we note that when $I_P\bk{\chi,s}$ admits a unique irreducible subrepresentation, $\tau$, then it is the unique subquotient of $I_P\bk{\chi,s}$ such that $r_T^G\tau$ contains an anti-dominant exponent.
	It follows that $\tau$ appears in $I_P\bk{\chi,s}$ with multiplicity $1$.

\vspace{0.3cm}

It remains to deal with the cases $\coset{4,-\frac{1}{2},1}$ and $\coset{4,-\frac{1}{2},3}$.
Both calculations are of the form suggested in \Cref{Subsec:Subs} as \textbf{Case III}. However, for $\coset{4,-\frac{1}{2},1}$ we show that $\pi=I_P\bk{\chi,-\frac{1}{2}}$ admits a unique irreducible subrepresentation while for $\coset{4,-\frac{1}{2},3}$ we show that the maximal semi-simple subrepresentation is of length 3.

$\bullet$ Consider $\pi=I_P\bk{\chi,-\frac{1}{2}}$, where $P=P_4$ and $\chi=\Id$.
Let $M'=M_6$ and $M''=M\cap M'=M_{1,2,3,5}$.
Furthermore, let $\Omega''=r_{M''}^M\Omega$, where $\Omega=\Omega_{M,\chi,s}$, and note that
\[
\lambda_0 = r_T^M\Omega = r_T^{M''}\Omega'' = 
\esixchar{-1}{-1}{-1}{2}{-1}{-1} .
\]
As explained in \Cref{Subsec:Subs},
\[
i_M^G\Omega \hookrightarrow i_{M'}^G \bk{ i_{M''}^{M'}\Omega''}.
\]
We now note that $i_{M''}^{M'}\Omega''$ is an irreducible representation of $M'$.
This follows from \cite[Theorem 5.3]{Ban2003}.
Alternatively, one can show, using \Cref{Subsec:Subs}, that $i_{M''}^{M'}\Omega''$ admits a unique irreducible subrepresentation.
On the other hand $i_{M''}^{M'}\Omega'' = \bk{i_{M''}^{M'}\Id}\otimes\Omega''$ and $i_{M''}^{M'}\Id$ is unitary.
It follows that $i_{M''}^{M'}\Omega''$ is semi-simple of length $1$ and hence it is irreducible.

We then note that $\lambda_{a.d.}\leq r_T^{M'} \bk{i_{M''}^{M'} \Omega''}$ and hence, by \Cref{Central_character_argument},
\[
i_{M''}^{M'} \Omega'' \hookrightarrow i_T^{M'} \lambda_{a.d.} .
\]
It follows that
\[
i_M^G\Omega \hookrightarrow i_{M'}^G \bk{i_T^{M'} \lambda_{a.d.}} \cong i_T^G \lambda_{a.d.}.
\]
Since $i_T^G \lambda_{a.d.}$ admits a unique irreducible subrepresentation, then so does $i_M^G\Omega$.

$\bullet$ Consider $\pi=I_P\bk{\chi,-\frac{1}{2}}$, where $P=P_4$ and $\chi$ is of order $3$.
We show that the maximal semi-simple subrepresentation of $\pi$ has length $3$.
In particular, this proves that $\pi$ is reducible.

The leading exponent in this case is
\[
\lambda_0 = \esixchar{-1}{-1}{-1}{2}{-1}{-1} + \chi\circ \esixchar{0}{0}{0}{1}{0}{0},
\]
while
\[
\lambda_{a.d.} = \esixchar{0}{0}{0}{-1}{0}{0} +\chi\circ \esixchar{1}{1}{1}{1}{1}{1}
\]
is an anti-dominant exponent in $r_T^G\pi$.

We note that $\lambda_{a.d.} = w_{42354}\cdot\lambda_0$ and that $N\bk{w_{42354},\lambda_0}$ is an isomorphism (as a composition of $5$ isomorphisms associated with simple reflections).
It follows that
\[
I_P\bk{\chi,-\frac{1}{2}} \hookrightarrow i_T^G \lambda_{a.d.} .
\]
One checks that
\[
\Stab_{W}\bk{\lambda_{a.d.}} = \gen{w_{3165}} \cong Z\rmod3\Z.
\]
Following the analysis in \cite{MR620252,MR1141803} we have
\[
i_T^{M_4}\lambda_{a.d.} = \bigoplus_{\xi\in\Z\rmod3\Z} \sigma_\xi .
\]
Since $i_{M_4}^G\sigma_\xi$ is a standard module for any $\xi\in\Z\rmod3\Z$, it admits a unique irreducible subrepresentation $\pi_\xi$.

So $\bigoplus_{\xi\in\Z\rmod3\Z} \pi_\xi$ is the maximal semi-simple subrepresentation of $i_T^G \lambda_{a.d.}$.
We prove that it is a subrepresentation of $\pi$.

Note that
\[
mult_{r_T^Gi_T^G\lambda_{a.d.}}\bk{\lambda_{a.d.}} = 3
\]
and
\[
mult_{r_T^G\pi_\xi}\bk{\lambda_{a.d.}} = 1
\]
for any $\xi\in\Z\rmod3\Z$.
Namely, these are the only subquotients of $i_T^G\lambda_{a.d.}$ containing the exponent $\lambda_{a.d.}$.
Since $mult_{r_T^G\pi}\bk{\lambda_{a.d.}} = 3$, $\bigoplus_{\xi\in\Z\rmod3\Z} \pi_\xi$ is the maximal semi-simple subrepresentation of $\pi$.

\end{proof}

\appendix
\appendixpage

\section{Branching Rules Coming from Small Levi Subgroups}
\label{App:Branching_Rules}

In this section, we describe the branching rules (see \Cref{Subsec:Irr_Test}) used in \Cref{Sec_degenerate_principal_series_E6}.
A branching rule stemming from a Levi subgroups $M$, whose derived subgroups $M^{der}$ is of type $X_n$ will be referred to as an $X_n$-branching rules.
We will only record the relevant results on the representation theory of $M$ and the branching rules following from it.
For further discussion on the representations of $GL_n$ for small $n$, see \cite{SegalRepresentations_of_GLn}.

We note that since the representations $I_P\bk{\chi,s}$ have cuspidal support on $T$, then the same holds for all the constituents of $r_M^G I_P\bk{\chi,s}$ for any standard Levi subgroup $M$ of $G$.
Namely, $r_M^G I_P\bk{\chi,s}$ does not contain supercuspidals.

\subsection{"Golden Rule" - Induction from the Trivial Character of the Torus}

Let $\pi$ be an irreducible representation of $G$, $\lambda\leq r_T^G \pi$ and let
\[
\Theta_\lambda = \set{\alpha\in\Delta \mvert \gen{\lambda,\check{\alpha}}=0} \subset\Delta.
\]
We note that $\lambda \perp \gen{\alpha\mvert \alpha\in{\Theta_\lambda}}$
and hence
\[
i_T^{M_{\Theta_\lambda}} \lambda = \bk{i_T^{M_{\Theta_\lambda}} \Id} \otimes \lambda .
\]
We also note that $i_T^{M_{\Theta_\lambda}} \Id$ is irreducible.
On the other hand, $\coset{r_T^{M_{\Theta_\lambda}} \bk{i_T^{M_{\Theta_\lambda}} \lambda}} = \Card{W_{M_{\Theta_\lambda}}} \times \coset{\lambda}$, and hence $i_T^{M_{\Theta_\lambda}} \lambda$ is the unique irreducible representation of $M_{\Theta_\lambda}$ such that $\lambda\leq r_T^{M_{\Theta_\lambda}}\pi$.
We conclude that
\begin{framed}
\begin{equation}
\label{Eq:Golden_rule}
\lambda \leq r_T^G\pi \Longrightarrow \Card{W_{M_{\Theta_\lambda}}} \times \coset{\lambda} \leq \coset{r_T^G\pi} .
\end{equation}
\end{framed}

\subsection{$A_1$-Branching Rules}

Let $M=M_{\set{\alpha}}$ be a Levi subgroup of $G$ whose derived subgroup $M^{der}$ is isomorphic to $SL_2$.
It follows that $i_T^M\lambda$ is irreducible if $\gen{\lambda,\check{\alpha}}\neq\pm 1$.
We then have the following $A_1$-branching rule:
\begin{framed}
\begin{equation}
\lambda\leq r_T^G\pi,\ \gen{\lambda,\check{\alpha}} \neq \pm 1 \Longrightarrow \coset{\lambda}+\coset{s_\alpha\cdot\lambda} \leq \coset{r_T^G\pi} .
\end{equation}
\end{framed}

\subsection{$A_2$-Branching Rules}
\label{Subsec:Branching_rule_Levi_of_type_A2}

Let $M=M_{\set{\alpha,\beta}}$ be a Levi subgroup of $G$ whose derived subgroup $M^{der}$ is isomorphic to $SL_3$.
Assume that $\gen{\lambda,\check{\alpha}}=\pm 1$ and $\gen{\lambda,\check{\beta}}=0$.
In such a case, $i_T^M\lambda$ has length two.
We write $i_T^M\lambda = \sigma_1 + \sigma_2$.
Up to conjugating induces, we have
\begin{align*}
& \coset{r_T^G\sigma_1} = \coset{\lambda} + \coset{s_\beta\cdot\lambda} + \coset{s_\alpha s_\beta\cdot\lambda} = 2\times\coset{\lambda} + \coset{s_\alpha\cdot\lambda} \\
& \coset{r_T^G\sigma_2} = \coset{s_\alpha\cdot\lambda} + \coset{s_\beta s_\alpha\cdot\lambda} + \coset{s_\alpha s_\beta s_\alpha\cdot\lambda} = 2\times \coset{s_\beta s_\alpha\lambda} + \coset{s_\alpha\cdot\lambda}
\end{align*}

We then have the following $A_2$-branching rule:
\begin{framed}
	\begin{equation}
	\lambda\leq r_T^G\pi,\ \gen{\lambda,\check{\alpha}} = \pm 1,\  \gen{\lambda,\check{\beta}} = 0 \Longrightarrow 2\times\coset{\lambda}+\coset{s_\alpha\cdot\lambda} \leq \coset{r_T^G\pi} .
	\end{equation}
\end{framed}

\subsection{$A_3$-Branching Rules}

Let $M=M_{\set{\alpha,\beta,\gamma}}$ be a Levi subgroup of $G$ whose derived subgroup $M^{der}$ is isomorphic to $SL_4$.
Assume that $\beta$ is a neighbor of both $\alpha$ and $\gamma$ in $Dyn\bk{G}$, the Dynkin diagram of $G$.
Further assume that $\gen{\lambda,\check{\alpha}}=1$, $\gen{\lambda,\check{\beta}}=0$ and $\gen{\lambda,\check{\gamma}}=-1$.
There exists a unique irreducible representation $\sigma$ of $M$ such that $\lambda\leq r_T^G\sigma$.
Furthermore, it holds that
\[
\coset{r_T^G\sigma} = 2\times\coset{\lambda} + \coset{s_\alpha\cdot\lambda} + \coset{s_\gamma\cdot\lambda} + 2\times \coset{s_\alpha s_\gamma\cdot\lambda} .
\]
We conclude the following $A_3$-branching rule:
\begin{framed}
	\begin{equation}
	\begin{array}{c}
	\lambda\leq r_T^G\pi,\ 
	\gen{\lambda,\check{\alpha}}=1,\ \gen{\lambda,\check{\beta}}=0,\ \gen{\lambda,\check{\gamma}}=-1 \\
	\Longrightarrow 2\times\coset{\lambda} + \coset{s_\alpha\cdot\lambda} + \coset{s_\gamma\cdot\lambda} + 2\times \coset{s_\alpha s_\gamma\cdot\lambda} \leq \coset{r_T^G\pi} .
	\end{array}
	\end{equation}
\end{framed}

\subsection{$D_4$-Branching Rules}

Let $M=M_{\set{\alpha,\beta,\gamma,\delta}}$ be a Levi subgroup of $G$ whose derived subgroup $M^{der}$ is isomorphic to $Spin_8$ or $SO_8$.
The representation theory of these groups was studied in \cite{MR2017065} and \cite{Gan-Savin-D4}.

Assume that $\beta$ is a neighbor of both $\alpha$, $\gamma$ and $\delta$ in $Dyn\bk{G}$ and let $w_0=s_{\beta}s_{\alpha}s_{\gamma}s_{\delta}s_{\beta}$.

Further assume that $\gen{\lambda,\check{\beta}}=-1$ and $\gen{\lambda,\check{\alpha}}=\gen{\lambda,\check{\gamma}}=\gen{\lambda,\check{\delta}}=0$.
There exists a unique irreducible representation $\sigma$ of $M$ such that $\lambda\leq r_T^G\sigma$.
Furthermore, it holds that
\[
\coset{r_T^G\sigma} = 
\suml_{w\in W_M^{M_{\set{\beta}},T} \setminus \set{1,s_\beta} }
\coset{w\cdot\bk{w_0\cdot\lambda}}
.
\]
We conclude the following $D_4$-branching rule:
\begin{framed}
	\begin{equation}
	\begin{array}{c}
	\lambda\leq r_T^G\pi,\ 
	\gen{\lambda,\check{\beta}}=-1,\ \gen{\lambda,\check{\alpha}}=\gen{\lambda,\check{\gamma}}=\gen{\lambda,\check{\delta}}=0 \\
	\Longrightarrow 
	\suml_{w\in W_M^{M_{\set{\beta}},T} \setminus \set{1,s_\beta} }
	\coset{w\cdot\bk{w_0\cdot\lambda}} \leq \coset{r_T^G\sigma} .
	\end{array}
	\end{equation}
\end{framed}


\section{Example of Application of Branching Rules to Prove the Irreducibility of $I_P\bk{\chi,s}$}
\label{App:Example_of_Irr_Cir}

In this section, we consider a detailed example of the application of branching rules in order to prove the irreducibility of $I_{P_1}\bk{\Id,-2}$.

We start by writing the initial exponent of $I_{P_1}\bk{\Id,-2}$:
\[
\lambda_0 = \esixcharclass{3}{-1}{-1}{-1}{-1}{-1} .
\]
The anti-dominant exponent in the $W$-orbit of $\lambda_0$ is
\[
\lambda_{a.d.} = \esixcharclass{-1}{0}{-1}{-1}{0}{-1} .
\]
Let $\pi$ be an irreducible representation of $G$ such that $\pi\leq I_{P_1}\bk{\Id,-2}$ and $\lambda_{a.d.}\leq r_T^G\pi$.
Note that $\dim_\C r_T^G I_{P_1}\bk{\Id,-2}=27$ and hence $\dim_\C r_T^G\pi\leq 27$.
In particular, if we show that $\dim_\C r_T^G\pi=27$, then it would follow that $\pi=I_{P_1}\bk{\Id,-2}$ is irreducible.

We apply branching rules as follows:
\begin{itemize}
	\item 
	Since $\Card{\Stab_W\bk{\lambda_{a.d.}}} = 4$, it follows from \Cref{Eq:Golden_rule} that $4\times\lambda_{a.d.} \leq r_T^G \pi$.
	
	\item Applying the $A_2$-branching rules on $\lambda_{a.d.}$ with respect to the Levi subgroups $M_{\set{\alpha_2,\alpha_4}}$ (or $M_{\set{\alpha_4,\alpha_5}}$) and $M_{\set{\alpha_5,\alpha_6}}$ yields:
	\[
	2\times \esixcharclass{-1}{-1}{-2}{1}{-1}{-1} + 2\times \esixcharclass{-1}{0}{-1}{-1}{-1}{1} \leq r_T^G\pi .
	\]
	
	\item
	Applying, on $\esixcharclass{-1}{-1}{-2}{1}{-1}{-1}$, the $A_1$-branching rule associated to $M_{\set{\alpha_3}}$ yields
	\[
	2\times \esixcharclass{-3}{-1}{2}{-1}{-1}{-1} \leq r_T^G\pi .
	\]
	
	\item 
	Applying the $A_1$-branching rule associated to $M_{\set{\alpha_1}}$ on this exponent yields
	\[
	2\times \esixcharclass{3}{-1}{-1}{-1}{-1}{-1} \leq r_T^G\pi .
	\]
	
	\item Applying on $\esixcharclass{-1}{0}{-1}{-1}{-1}{1}$ the $A_2$-branching rule associated to $M_{\set{\alpha_2,\alpha_4}}$ yields
	\[
	\esixcharclass{-1}{-1}{-2}{1}{-2}{1} \leq r_T^G \pi .
	\]
	
	
	\item 	We now apply a sequence of $A_1$-branching rules associated with entries which are not $0$ or $\pm 1$ on $\esixcharclass{-1}{-1}{-2}{1}{-2}{1}$ and subsequent exponents.
	This yields
	
	\begin{align*}
	&\esixcharclass{-3}{-1}{2}{-1}{-2}{1}
	+ \esixcharclass{-1}{-1}{-2}{-1}{2}{-1} \\
	&+\esixcharclass{-3}{-1}{2}{-3}{2}{-1}
	+ \esixcharclass{3}{-1}{-1}{-1}{-2}{1} \\
	&+\esixcharclass{3}{-1}{-1}{-3}{2}{-1} 
	+ \esixcharclass{3}{-4}{-4}{3}{-1}{-1} \\
	&+\esixcharclass{-3}{-4}{-1}{3}{-1}{-1} 
	+ \esixcharclass{-1}{-4}{4}{-1}{-1}{-1} \\
	&+\esixcharclass{3}{4}{-4}{-1}{-1}{-1}
	+ \esixcharclass{-1}{4}{4}{-5}{-1}{-1} \\
	&+\esixcharclass{-3}{4}{-1}{-1}{-1}{-1}
	+ \esixcharclass{-1}{-1}{-1}{5}{-6}{-1} \\
	&+\esixcharclass{-1}{-1}{-1}{-1}{6}{-7} 
	+ \esixcharclass{-1}{-1}{-1}{-1}{-1}{7} \leq r_T^G\pi
	\end{align*}

\end{itemize}
We conclude that
\begin{align*}
& 4\times \lambda_{a.d.} + 2\times \esixcharclass{-1}{-1}{-2}{1}{-1}{-1} + 2\times \esixcharclass{-1}{0}{-1}{-1}{-1}{1} \\
&+2\times \esixcharclass{-3}{-1}{2}{-1}{-1}{-1} 
+2\times \esixcharclass{3}{-1}{-1}{-1}{-1}{-1} \\
&+\esixcharclass{-1}{-1}{-2}{1}{-2}{1} + \esixcharclass{-3}{-1}{2}{-1}{-2}{1} \\
&+\esixcharclass{-1}{-1}{-2}{-1}{2}{-1}
+ \esixcharclass{-3}{-1}{2}{-3}{2}{-1} \\
&+\esixcharclass{3}{-1}{-1}{-1}{-2}{1}
+ \esixcharclass{3}{-1}{-1}{-3}{2}{-1} \\
&+\esixcharclass{3}{-4}{-4}{3}{-1}{-1} 
+ \esixcharclass{-3}{-4}{-1}{3}{-1}{-1} \\
&+\esixcharclass{-1}{-4}{4}{-1}{-1}{-1}
+ \esixcharclass{3}{4}{-4}{-1}{-1}{-1} \\
&+\esixcharclass{-1}{4}{4}{-5}{-1}{-1}
+ \esixcharclass{-3}{4}{-1}{-1}{-1}{-1} \\
&+\esixcharclass{-1}{-1}{-1}{5}{-6}{-1}
+ \esixcharclass{-1}{-1}{-1}{-1}{6}{-7} \\
&+\esixcharclass{-1}{-1}{-1}{-1}{-1}{7} \leq r_T^G\pi
\end{align*}
Since we have proven that $27\leq \dim_\C r_T^G\pi$, it follows that $\pi=I_{P_1}\bk{\Id,-2}$ is irreducible.

\todonum{Maybe add more examples? Maybe find a simple example where one needs to "round-up" the multiplicity (say, from 3 to 4 so it will be even)?}


\bibliographystyle{plain}
\bibliography{bib}


\end{document}